\documentclass[11pt]{article}
\usepackage[latin1]{inputenc}
\usepackage{amsmath,amsthm,amssymb}
\usepackage{amsfonts}
\usepackage{amsmath,amsthm,amssymb,amscd}
\usepackage{latexsym}
\usepackage{color}
\usepackage{graphicx}
\usepackage{mathrsfs}

\makeatletter \@addtoreset{equation}{section} \makeatother

\newcommand{\sgn}{\mbox{sgn}}
\newtheorem{theorem}{Theorem}[section]
\newtheorem{definition}{Definition}[section]

\newtheorem{lemma}{Lemma}[section]
\newtheorem{remark}{Remark}[section]

\begin{document}	 	
\title{\bf \sc On critical variable-order Kirchhoff type problems with variable singular exponent} 
\author{\bf \sc Jiabin Zuo$^{a}$,	Debajyoti Choudhuri$^b$	and Du\v{s}an D.  Repov\v{s}$^{c,d,e}$\footnote{Corresponding author: dusan.repovs@guest.arnes.si}\\
\small{$^a$School of Mathematics and Information
Science, Guangzhou University, Guangzhou, 510006, China. {\sl zuojiabin88@163.com}}\\
\small{$^b$Department of Mathematics, National Institute of Technology, Rourkela, 769008, Odisha, India. {\sl dc.iit12@gmail.com}}\\
\small{$^c$Faculty of Education, University of Ljubljana,  Ljubljana, 1000, Slovenia. {\sl dusan.repovs@pef.uni-lj.si}}\\
\small{$^d$Faculty of Mathematics and Physics, University of Ljubljana, Ljubljana, 1000, Slovenia. {\sl dusan.repovs@fmf.uni-lj.si}}\\
\small{$^e$Institute of Mathematics, Physics and Mechanics, Ljubljana, 1000, Slovenia. {\sl dusan.repovs@guest.arnes.si}}}
\maketitle
\begin{abstract}
  We establish a continuous embedding $W^{s(\cdot),2}(\Omega)\hookrightarrow L^{\alpha(\cdot)}(\Omega)$, where the variable exponent $\alpha(x)$ can be close to the critical exponent $2_{s}^*(x)=\frac{2N}{N-2\bar{s}(x)}$, with $\bar{s}(x)=s(x,x)$ for	all
$x\in\bar{\Omega}$. Subsequently, this continuous embedding is used to prove the multiplicity of solutions for critical nonlocal degenerate Kirchhoff problems with a variable singular exponent. Moreover, we also obtain the uniform $L^{\infty}$-estimate of these infinite solutions by a  bootstrap argument.
\begin{flushleft}
{\sl Keywords and phrases}: Fractional Laplacian of variable-order; Continuous embedding; Genus; Symmetric mountain pass theorem; Variable singular exponent.\\
{\sl Math. Subj. Classif.}:~35R11, 35J75, 35J60, 46E35.
\end{flushleft}
\end{abstract}
		 
\section{Introduction}

A strong motivation behind the partial differential equation considered in this paper comes from the following problem, studied by Bisci et al. \cite{22}:
\begin{equation}\label{local_problem}
	\left\{\begin{array}{cl}
	&\displaystyle-\mathfrak{M}\left(\int_{\Omega}|\nabla u|^2dx\right)\Delta u=\frac{\lambda}{u^{\gamma}}+u^{2^{*}-1}~~~~\text{in}~\Omega,\\
	&u>0~~~~~~~~~~~~~~~~~~~~~~~~~~~~~~~~~~~~~~~~~~~~~\text{in}~\Omega,\\
		&u=0~~~~~~~~~~~~~~~~~~~~~~~~~~~~~~~~~~~~~~~~~~~~~\text{on}~\partial\Omega.\\
	\end{array}\right.
	\end{equation}

This  is a typical Kirchhoff problem. Here, $\mathfrak{M}:\mathbb{R}_0^+\rightarrow\mathbb{R}_0^+$ is the Kirchhoff function, usually of the  form			
		 \begin{align}\label{kirchhoff_model}\mathfrak{M}(t)&=a+bt^{\theta-1},~a,b\geq 0,~a+b>0,~\theta\geq 1.\end{align}
		In particular, when $\mathfrak{M}(t)\geq c>0$ for some $c$ and all $t\in\mathbb{R}_0^+$, Kirchhoff problem is said to be non-degenerate. On the other hand, when $\mathfrak{M}(0)=0$ and $\mathfrak{M}(t)>0$ for all $t\in\mathbb{R}^+$,   Kirchhoff problem is said to be degenerate.
		
		Recenty, this kind of nonlocal problems have been widely studied. The readers can be referred to
		\cite{17,18,19,	21,20} for various kinds of Kirchhoff problems with $\mathfrak{M}$ as in \eqref{kirchhoff_model}, which are driven by a Laplacian and involve a singular term  $u^{-\gamma}$. Lei et al.~\cite{17}
		obtained the multiplicity of solutions for Kirchhoff problems like \eqref{local_problem}. Liu and Sun~\cite{21} studied a class of singular Kirchhoff problems with a  Hardy potential with the help of the Nehari manifold technique.		
		Furthermore, Liao et al. \cite{19} used the Nehari manifold approach
		 for the subcritical Kirchhoff problem with an additional singular term. A noteworthy mention at this stage is the work due to Gu et al. \cite{18} where the existence and the uniqueness of solution is obtained by a minimization method, to a singular problem involving a non-positive critical nonlinearity.
		
		Following Lei et al. \cite{17},  Liu et al.~\cite{20} obtained two solutions of  Kirchhoff problems with a critical term and a singular term $|x|^{-\beta}u^{-\gamma}$. Barrios et al.~\cite{4} investigated the following equation
		\begin{align}\label{barrios_prob}
		\left\{\begin{array}{cl}
	&(-\Delta)^su=\lambda\frac{f(x)}{u^{\gamma}}+Mu^p~~~~~\text{in}~\Omega,\\
		&u>0~~~~~~~~~~~~~~~~~~~~~~~~~~~~~\text{in}~\Omega,\\
		&u=0~~~~~~~~~~~~~~~~~~~~~~~~~~~~~\text{in}~\mathbb{R}^N\setminus\Omega\\
	\end{array}\right.
		\end{align}
		with $f$ a nonnegative function, $p>1$, and $M\in\{0,1\}$.
		They studied  the existence and multiplicity of solutions by employing the  Sattinger method and the sub/supersolution technique. Canino et al.~\cite{Canino2017}  generalized the result by Chen et al. \cite[ Section 3]{4'} to the case of $p$-fractional Laplacian $(-\Delta_p)^s$. We draw the attention of the reader to
		\cite{2,Sekhar2020}
		(not restricted to only these) for existence results and \cite{pos_ghosh,24,25,29,5} for the multiplicity results.

		Off-late, from a scientific point of view, fractional Sobolev spaces and related non-local problems have attracted the attention of many scholars because they occur naturally in many fields, such as electrorheological fluids and image processing (cf. Barrios et al. \cite{4} and the references therein). Readers who are interested to know the physical motivation behind the study of elliptic problems involving Kirchhoff operator can refer to Carrier \cite{carrier}.
		In fact, there are only a few papers on the $p(x)$-Laplace operator
involving singular nonlinearity and some of which can be found in the articles \cite{avci1,Sekhar2020} and the references therein.
		
At the end of the 20th century, the fractional variable derivatives proposed by two mathematicians, Ross and Samko \cite{SGS}, first appeared in nonlinear diffusion processes. At that time, scholars discovered that the temperature change of some objects with a reaction diffusion process can be better expressed by the variable derivative of a nonlocal integral-differential operator. Therefore, gradually, the problem of fractional variable-order   became favored by many authors.

Xiang et al. \cite{mingqi1} studied an elliptic problem driven by a fractional Laplace operator involving variable-order and proved the existence of two solutions using the Mountain pass theorem and Ekeland's variational principle. However, their key result  is that there exist infinitely many solutions to the limit problem when a parameter $\lambda$  goes to $\infty$. Xiang et al.~\cite{mingqi2} considered a variable-order fractional Kirchhoff-type problem which could be degenerate, they proved that there exists two nonnegative solutions using the Nehari manifold approach.				 Furthermore, they also  verified the existence of infinitely many solutions with the help of a symmetric critical point theorem.
		Wang and Zhang~\cite{wang_zhang_1} studied a variable-order fractional Kirchhoff-type problem involving variable exponents and also obtained the existence of infinitely many solutions. 
		
Further papers that can be consulted for the study of Kirchhoff problem are Fiscella \cite{fisc1}, Ghosh and Choudhuri~\cite{sekhar1}, Khiddi and Sbai \cite{khiddi1}. However, as far as we know, none of the cited works  addressed the problem of a variable-order Laplacian with a singularity of variable exponent. The novelty of our present work lies in the treatment of the problem with a variable critical nonlinearity and the estimate of the boundedness of infinitely many  solutions.
			
\section{Preliminaries}		 

  Let $X$ be a normed linear space
		$$X=\left\{u:\mathbb{R}^N\rightarrow\mathbb{R}~\text{measurable}:~u|_{\Omega}\in L^2(\Omega)~\text{and}~\left(\iint_{Q}\frac{|u(x)-u(y)|^2}{|x-y|^{N+2s(x,y)}}dxdy\right)^{\frac{1}{2}}<\infty\right\}$$
		 with the Gagliardo norm
		$$\|u\|_{X}=\|u\|_{2}+\left(\iint_{Q}\frac{|u(x)-u(y)|^2}{|x-y|^{N+2s(x,y)}}dxdy\right)^{\frac{1}{2}},$$
where
$$\Omega\subset\mathbb{R}^N, \quad  Q=\mathbb{R}^{2N}\setminus((\mathbb{R}^N\setminus\Omega)\times(\mathbb{R}^N\setminus\Omega))$$
and $\|u\|_2$ denotes the $L^2$-norm of $u$ over $\Omega$.
		
		We shall frequently use the following subspace $X_0$ of $X$
		$$X_0=\left\{u\in X:u=0~\text{in}~\mathbb{R}^N\setminus\Omega\right\}$$
		 with the norm
		$$\|u\|_{X_0}=\left(\iint_{Q}\frac{|u(x)-u(y)|^{2}}{|x-y|^{N+2s(x,y)}}dxdy\right)^{\frac{1}{2}}.$$		

		Our workspace 
		$(X_0,\|\cdot\|_{X_0})$ will be a Hilbert space with the norm $\|\cdot\|_{X_0}$ 
		induced by the inner product 
		$$
		\langle \psi,\varphi\rangle=\iint_{Q}\frac{(\psi(x)-\psi(y))(\varphi(x)-\varphi(y))}{|x-y|^{N+2s(x,y)}}dxdy,
		\quad
		\hbox{for all}
		\quad \psi, \varphi\in X_0.$$ 		
		
		Our workspace $X_{0}$ will also be a Banach space with reflexivity and separability, see~\cite[Lemma 2.3]{m4}.
		
		We conclude this section by defining the best Sobolev constant
		\begin{align}\label{sobolev_constant}
		S&=\underset{u\in X_0\setminus\{0\}}{\inf}\frac{\iint_{Q}\frac{|u(x)-u(y)|^{2}}{|x-y|^{N+2s(x,y)}}dxdy}{\|u\|_{\alpha(\cdot)}^2}.
		\end{align}
		\begin{remark}\label{space_equiv}
			
		\begin{enumerate}
				\item Note that
				$X,  X_0\subset W^{s(\cdot),2}(\Omega)$,   where $W^{s(\cdot),2}(\Omega)$ is the usual fractional variable-order
				Sobolev space  with the so-called Gagliardo norm
				 $$\|u\|_{W^{s(\cdot),2}(\Omega)}=\|u\|_{2}+\left(\iint_{\Omega\times\Omega}\frac{|u(x)-u(y)|^2}{|x-y|^{N+2s(x,y)}}dxdy\right)^{\frac{1}{2}}.$$
				
				\item Any uniform constant appearing in a well-known embedding will be represented by one of the symbols $C$, $C'$, $C''$, etc.
				with a prefix and/or suffix.
			\end{enumerate}
		\end{remark}		

	\begin{remark}\label{FA_Res}
		For all  fundamental results in functional analysis we direct the readers to
		Br\'{e}zis \cite{6}. For more details on fractional Sobolev spaces we refer to Di Nezza et al. \cite{11}.
		For further details on Lebesgue spaces with variable exponents we refer to R\u{a}dulescu and Repov\v{s} \cite{RD}.
	\end{remark}
	
	\section{Statement of the problem and the  main results}
	  We  investigate the existence of infinite solutions to the following variable-order problem with critical variable exponent growth
	\begin{equation}\label{problem}
	\left\{\begin{array}{cl}
	&\left(\displaystyle\iint_{\mathbb{R}^{2N}}\frac{|u(x)-u(y)|^2}{|x-y|^{N+2s(x,y)}}dxdy\right)^{\theta-1}(-\Delta)^{s(.)} u=\lambda\dfrac{u}{|u|^{\gamma(x)+1}}+|u|^{\alpha(x)-2}u~~~~\text{in}~\Omega,\\
	&			u=0~~~~~~~~~~~~~~~~~~~~~~~~~~~~~~~~~~~~~~~~~~~~~~~~~~~~~~~~~~~~~~~~~~~~~~~~~~~~~~~~~~~~~~~~\text{in}~\mathbb{R}^N\setminus\Omega,
	\end{array}\right.
	\end{equation}
	where $\Omega\subset\mathbb{R}^N$ with $ N\geq2$ is a smooth bounded domain, $\lambda>0$, $s(\cdot):\mathbb R^{2N}\to(0,1)$ is a continuous function such that 
	$$0<s^-=\underset{{(x,y)\in\overline{\Omega}\times\overline{\Omega}}}{\inf}\{s(x,y)\}\leq s(x,y)\leq s^+=\underset{{(x,y)\in\overline{\Omega}\times\overline{\Omega}}}{\sup}\{s(x,y)\}<1,$$
	 $\gamma(\cdot): \Omega\to(0,1)$ is a continuous function such that 
	 $$0<\gamma^-=\underset{{x\in\overline{\Omega}}}{\inf}\{\gamma(x)\}\leq \gamma(x)\leq \gamma^+=\underset{{x\in\overline{\Omega}}}{\sup}\{\gamma(x)\}<1,$$
	   $\alpha(\cdot): \Omega\to \mathbb{R}$ is a continuous function such that 
	   $$1-\gamma^-<1<2\theta<\alpha^-=\underset{x\in\overline{\Omega}}{\inf}\{\alpha(x)\}\leq \alpha(x)\leq \alpha^+=\underset{x\in\overline{\Omega}}{\sup}\{\alpha(x)\}\leq 2_{s}^*(x),$$ 
	   and $(-\Delta)^{s(\cdot)}$ is the fractional Laplace operator of variable-order $s(\cdot)$.	
	The fractional Laplace operator of a variable-order $s(\cdot)$ is defined by
	\begin{eqnarray}
	(-\Delta)^{s(\cdot)}u(x)&=&C_{N,s(\cdot)}\lim_{\epsilon\rightarrow 0}\int_{\mathbb{R}^N\setminus B_{\epsilon}(x)}\frac{(u(x)-u(y))}{|x-y|^{N+2s(x,y)}}dy, x\in\mathbb{R}^N,\nonumber
	\end{eqnarray}
	where $C_{N,s(\cdot )}>0$ is an explicit constant.
	
	Our main results are as follows.
	\begin{theorem}\label{embedding}
		Let $\Omega\subset\mathbb{R}^N$ be a Lipschitz bounded domain, $N>2s(\cdot)$ in $\bar{\Omega}\times\bar{\Omega}$,
		and
		 $\alpha(\cdot):\bar{\Omega}\to\mathbb R$ a \textit{continuous} function satisfying the following conditions:
		\\
		$(A_{1})$ $1<\alpha^{-}\leq \alpha^{+}$;
		\\
		$(A_{2})$  \textit{there exists $\varepsilon=\varepsilon(x)>0$ such that
			\begin{equation}\label{ps}
			 \underset{y\in\Omega_{x,\epsilon}}{\sup}\{\alpha(y)\}\leq\frac{2N}{N-2\underset{(y,z)\in\Omega_{x,\varepsilon}\times\Omega_{x,\epsilon}}{\inf} \{s(y,z)\}},
			\end{equation}
			where $\Omega_{z,\varepsilon}=B_{\varepsilon}(z)\bigcap\Omega,$ for  $z\in\Omega$.} \\
		Moreover, let $s(\cdot):\mathbb R^{2N}\to(0,1)$ be \textit{continuous} and 
		assume that
		\\
		$(H_{1})$ $0<s^{-}\leq s^{+}<1$;
		\\
		$(H_{2})$ \textit{$s(\cdot)$ is symmetric, that is, $s(x,y)=s(y,x),$ for any $(x,y)\in\mathbb R^{2N}$.}
		\\		Then there exists a constant $C=C(N,s,\alpha,\Omega)$ such that for every $u\in X,$ the following holds
		$$\|u\|_{\alpha(\cdot)}\leq C\|u\|_{W^{s(\cdot),2}(\Omega)},$$
		that is, the embedding $W^{s(\cdot),2}(\Omega)\hookrightarrow L^{\alpha(x)}(\Omega)$ is continuous.
	\end{theorem}
	
	\begin{theorem}\label{mainthm1}
		There exists $\lambda_0>0$ such that for any $\lambda\in(0,\lambda_0),$  problem  \eqref{problem} has infinitely many solutions in $X_0$.	
	\end{theorem}
	\begin{theorem}\label{bdd_thm}
		Let $0\leq u\in X_0$ be a weak solution to problem  \eqref{problem}. Then $u\in L^{\infty}(\Omega)$.
	\end{theorem}
	
\subsection{Weak solutions}\label{3.1}
	
  Associated to  problem \eqref{problem}, the energy functional $I_\lambda\colon X_0\rightarrow\mathbb{R}$   is defined by
\begin{align}\label{energy p}
I_{\lambda}(u) =
\frac{1}{2\theta}\|u\|_{X_0}^{2\theta}-\int_{\Omega}\frac{\lambda}{1-\gamma(x)}|u|^{1-\gamma(x)}dx-\int_{\Omega}\frac{1}{\alpha(x)}|u|^{\alpha(x)}dx.
\end{align}

Having defined  $I_{\lambda}$, we now define the sense in which the solution to problem \eqref{problem} will be considered.
\begin{definition}\label{defn_weak_soln_main_prob}$u\in X_0$ is said to be a {\it weak solution} to  problem \eqref{problem} if $|u|^{-\gamma(x)}\varphi \in L^1(\Omega)$ and
	\begin{align}\label{weak_soln}
	\|u\|_{X_0}^{2(\theta-1)}\langle u,\varphi\rangle&=\int_{\Omega}\lambda|u|^{-\gamma(x)-1}u\varphi dx+\int_{\Omega}|u|^{\alpha(x)-2}u\varphi dx
	\end{align}
	for every $\varphi\in X_0$. Here,
	$$\langle u,\varphi\rangle=\iint_{Q}\frac{(u(x)-u(y))(\varphi(x)-\varphi(y))}{|x-y|^{N+2s(x,y)}}dxdy.$$
\end{definition}

Note that the functional $I_{\lambda}$ is continuous in $X_0$, but it fails to be $C^1$ in $X_0$. This is  bad news as far as the application of theorems in the variational analysis is concerned because, most of them demand the energy functional to be $C^1$ over its domain of definition.
Thus, inspired by the {\it cutoff} technique introduced by Clark and Gilbarg \cite{clark1972variant} and later used by Gu et al.\cite{Gu2018aml}, we shall apply it to  problem \eqref{problem}.

Choose $l$ to be sufficiently small and  define  an even  \textit{cutoff} 
$C^{1}$-function $\eta:\mathbb{R}\rightarrow\mathbb{R}^+$ such that $0\leq\eta(t)\leq1$ and
$$\eta(t)=\begin{cases}
1, ~\text{if}~ |t|\leq l\\
\eta ~\text{is decreasing, if}~ l\leq t\leq 2l\\
0,~\text{if}~ |t|\geq 2l.
\end{cases}$$

Set $$f(u(x)):=|u(x)|^{\alpha(x)-2}u(x) \quad \hbox{for all} \quad u\in X_0.$$ 

Clearly, we have $f(-u(x))=-f(u(x))$. Now consider the following {\it cutoff} problem:
\begin{equation}\label{main3}
\left\{\begin{array}{cl}
&\left(\displaystyle\iint_{\mathbb{R}^{2N}}\frac{|u(x)-u(y)|^2}{|x-y|^{N+2s(x,y)}}dxdy\right)^{\theta-1}(-\Delta)^{s(.)} u=\lambda\dfrac{u}{|u|^{\gamma(x)+1}}+\tilde{f}(u)~~~~\text{in}~\Omega,\\
&			u=0~~~~~~~~~~~~~~~~~~~~~~~~~~~~~~~~~~~~~~~~~~~~~~~~~~~~~~~~~~~~~~~~~~~~~~~~~~~~~~~~~\text{in}~\mathbb{R}^N\setminus\Omega,
\end{array}\right.
\end{equation}
where
$$\tilde{f}(u(x))=f(u(x))\eta\left(\frac{\|u\|_{X_0}^2}{2}\right).$$

We further define $$F(u(x))=\frac{|u(x)|^{\alpha(x)}}{\alpha(x)},
\quad
\tilde{F}(u(x))=F(u(x))\eta\left(\frac{\|u\|_{X_0}^2}{2}\right).$$
\begin{remark}\label{key_obs}It is easy to see that if $u$ is a weak solution to problem \eqref{main3} with $\frac{\|u\|_{X_0}^2}{2}\leq l$, then $u$ is also a weak solution to problem \eqref{problem}. 
\end{remark}

\begin{remark}\label{delta_0}
	We shall  use
	\begin{equation}\label{delta}\delta=\begin{cases}
	\gamma^{+}, & \|u^{1-\gamma(x)}\|_{\frac{1}{1-\gamma(x)}} < 1\\
	\gamma^{-}, & \|u^{1-\gamma(x)}\|_{\frac{1}{1-\gamma(x)}}  > 1
	\end{cases}\end{equation}
	whenever we shall  estimate the singular term.
	Throughout this article, wherever necessary, any subsequence of a given sequence will be represented by the symbol of the sequence itself.
\end{remark}
\begin{remark}\label{Not1}
	For better understanding of some of the notations, we define 
	$$\|u\|_{X_0}=\|U\|_{L^2(k(\cdot)dxdy)}
	\quad
	\hbox{ for all}
	\quad u\in X_0,$$ where $$
	U(x,y)=u(x)-u(y),
	\quad
	k(\cdot)dxdy=k(x,y)dxdy=\frac{1}{|x-y|^{N+2s(x,y)}}dxdy$$
	 and
	\begin{align}\label{notation1}\|U\|_{L^2(k(\cdot)dxdy)}=\int_{\Omega\times\Omega}|U(x,y)|^2k(x,y)dxdy.\end{align}
\end{remark}
The energy functional $\bar{I}_{\lambda}:X_0\rightarrow\mathbb{R}$, associated to  problem \eqref{main3} is defined as follows:
$$\bar{I}_{\lambda}(u)=\frac{1}{2\theta}\|u\|_{X_0}^{2\theta}-\int_{\Omega}\frac{\lambda}{1-\gamma(x)}|u|^{1-\gamma(x)}dx-\int_{\Omega}\tilde{F}(u)dx.$$

Furthermore, the Fr\'{e}chet derivative of $\int_{\Omega}\tilde{F}(u)dx$ is
\begin{align}\label{frech_der}
\eta\left(\frac{\|u\|_{X_0}^2}{2}\right)\int_{\Omega}\tilde{f}(u)\varphi dx+\eta'\left(\frac{\|u\|_{X_0}^2}{2}\right)\langle u,\varphi \rangle\int_{\Omega}\tilde{F}(u)dx~\text{for every}~\varphi\in X_0.
\end{align}

However, by  Remark \ref{key_obs}, since we are interested in those solutions which obey $\frac{\|u\|_{X_0}^2}{2}\leq l$,  the Fr\'{e}chet derivative of $\int_{\Omega}\tilde{F}(u)dx$ can be considered to be of the form
\begin{align}\label{frech_der_1}
\int_{\Omega}\tilde{f}(u)\varphi dx~\text{for every}~\varphi\in X_0.
\end{align}

 Note that $\bar{I}_{\lambda}$ is an even functional over $X_0$.
Finally, we give the definition of weak solutions to problem \eqref{main3}.
\begin{definition}\label{cutoff_weak_soln}
	$u\in X_0$ is said to be a weak solution to problem \eqref{main3} if $|u|^{-\gamma(x)}\varphi \in L^1(\Omega)$ and
	\begin{align}\label{weak_soln_mod_prob}
	\|u\|_{X_0}^{2(\theta-1)}\langle u,\varphi\rangle&=\lambda\int_{\Omega} {|u|}^{-\gamma(x)-1}u\varphi dx+\int_{\Omega}\tilde{f}(u)\varphi dx
	\end{align}
	for every $\varphi\in X_0$.
\end{definition}	
	
\subsection{Existence of nontrivial solutions}

  The proof of Theorem \ref{mainthm1} will require several lemmas which we  now state and prove in this subsection. In order to establish the existence of infinitely many solutions to
  problem  \eqref{main3}, one needs a prior knowledge of whether at least one solution to problem \eqref{main3} exists or not.  Lemmas \ref{thm1} and \ref{thm2} will establish the existence of a nontrivial solution for problem \eqref{main3} in $X_0$ at which $\bar{I}_{\lambda}$ attains this infimum. This infimum incidentally has been found to be negative, in a small neighbourhood of zero.

 We at first notice that since $\bar{I}_{\lambda}\in C^0(X_0,\mathbb{R})$ and hence the following minimization problem can be considered
\begin{align}\label{min_prob}m_{\lambda}&=\underset{u\in \overline{B_r}}{\min}\bar{I}_{\lambda}(u)\end{align} for some $r>0$.
\begin{lemma}\label{thm1}
	There exist $r\in (0,1)$, $\lambda_r>0$ and $\beta_r>0$ such that $\bar{I}_{\lambda}(u)>\beta_r$ for all $\lambda\in (0,\lambda_r)$ and $u\in X_0$ with $\|u\|_{X_0}=r$. Let $$m_{\lambda}=\underset{u\in \overline{B_r}}{\inf}\bar{I}_{\lambda}(u).$$
	Then $m_{\lambda}<0$ for all $\lambda\in (0,\lambda_r]$.
\end{lemma}
\begin{proof}
	Let $\lambda>0$. Clearly, since $\bar{I}_{\lambda}$ is continuous and is defined at every point in $X_0$, an infimum of $\bar{I}_{\lambda}$ exists in any closed and bounded neighbourhood of $0\in X_0$. On using the H\"{o}lder inequality and the continuous embedding of $X_0$ in $L^1(\Omega)$, we have for any $u\in X_0,$
	\begin{eqnarray}\label{eqn0}
	\int_{\Omega}|u|^{1-\gamma(x)}dx&\leq &C\|u\|_{X_0}^{1-\delta},
	\end{eqnarray}
	where $C>0$.  It follows from Theorem \ref{embedding} that
	 $$\frac{1}{2\theta}\|u\|_{X_0}^{2\theta}-\int_{\Omega}\frac{1}{\alpha(x)}\eta\left(\frac{\|u\|_{X_0}^2}{2}\right)|u|^{\alpha(x)}dx\geq\frac{1}{2\theta}\|u\|_{X_0}^{2\theta}-\frac{C'}{\alpha^-}\|u\|_{X_0}^{\alpha^-}>0$$ for every $u\in \partial B_r,$ provided $0<r<1$ is small enough.
	 
	  Fix such an $r>0$,  taking $\lambda>0$ small enough, other negative terms in $\bar{I}_{\lambda}$ may be made arbitrarily small. Therefore we
	can
	 find $r$ and $\lambda$ such that
	$\underset{u\in\partial B_r}\min\{\bar{I}_{\lambda}(u)\}>0.$
	Thus there exists  $\beta_r>0$ such that $\bar{I}_{\lambda}(u)>\beta_r$ for $u\in X_0$ with
	 $\|u\|_{X_0}=r$.
	Moreover, since $\bar{I}_{\lambda}(tu)<0$ for $t$ small enough, we have
	$m_{\lambda}<0.$
	Therefore, say, $m_{\lambda}=\underset{u\in \bar{B}_r}\inf\{\bar{I}_{\lambda}(u)\}<0$.
\end{proof}

	  We now prove that for a finite range of $\lambda$, a solution to  problem \eqref{main3} exists.
	\begin{lemma}\label{thm2}
		Let $\lambda_r$ be given as in  Lemma~\ref{thm1}. Then for any $\lambda\in(0,\lambda_r)$, problem  \eqref{main3} has a solution $u_0\neq 0\in X_0$ with $\bar{I}_{\lambda}(u_0)=m_{\lambda}<0$.
	\end{lemma}
	\begin{proof}
	First, we show that there exists $u_0\in \bar{B}_r$ such that
		\begin{align}\label{aux0}
		\begin{split}
		\underset{n\rightarrow\infty}{\lim}\bar{I}_{\lambda}(u_n)&=\bar{I}_{\lambda}(u_0)=m_{\lambda}.
		\end{split}
		\end{align}
		
		Let $\{u_n\}$ be a minimizing sequence of $\bar{I}_{\lambda}$ chosen from the closed ball of radius $r$, i.e. $\bar{B}_r$. Hence it is a bounded sequence in $X_0$. Thus by the reflexivity of $X_0$ and Egoroff's theorem, there exists a subsequence $\{u_n\}$, such that
		\begin{align}\label{conv1}
		\begin{split}
		u_n&\rightharpoonup u_0~\text{in}~X_0,\\
		u_n&\rightarrow u_0~\text{in}~L^{p(x)}(\Omega)~\text{for any}~1<p(x)<2_{s^-}^*,\\
		u_n&\rightarrow u_0~\text{a.e. in}~\Omega,
		\end{split}
		\end{align}
		where $u_0\in \bar{B}_r$. Since $0<\gamma^-\leq\gamma^+<1$, we have by
		H\"{o}lder's inequality that for any $n$
		\begin{align}\label{sing_ineq}
		\begin{split}
		 \left|\int_{\Omega}|u_n|^{1-\gamma(x)}dx-\int_{\Omega}|u_0|^{1-\gamma(x)}dx\right|&\leq\int_{\Omega}\left||u_n|^{1-\gamma(x)}-|u_0|^{1-\gamma(x)}\right|dx 
		 \\
		&\leq\int_{\Omega}\left|u_n-u_0\right|^{1-\gamma(x)}dx 
				\leq C \|u_n-u_0\|_2^{1-\gamma^+}.
		\end{split}
		\end{align}
		
		Thus, using \eqref{conv1} in \eqref{sing_ineq}, we obtain
		\begin{align}\label{conv2}
		\begin{split}
		\underset{n\rightarrow\infty}{\lim}\int_{\Omega}|u_n|^{1-\gamma(x)}dx&=\int_{\Omega}|u_0|
		^{1-\gamma(x)}dx.\end{split}
		\end{align}
		
		Now, consider $v_n=u_n-u_0$. Then by  notation \eqref{notation1} in Remark \ref{Not1} and the Br\'{e}zis-Lieb Lemma (\cite{brezis1983relation}),
		\begin{eqnarray}\label{brez-lieb}
		\|U_n\|_{L^2(k(\cdot)dxdy)}&=&\|V_n\|_{L^2(k(\cdot)dxdy)}^2+\|U_0\|_{L^2(k(\cdot)dxdy)}^2+o(1),
		\end{eqnarray}
		as $n\rightarrow\infty$. Since $\{u_n\}\subset \bar{B}_r$,  by \eqref{brez-lieb} we have that $u_0, v_n\subset \bar{B}_r$. Now,
		by
		 Lemma \ref{thm1}, for every $u\in X_0$ with $\|u\|_{X_0}=r<1$, we obtain
		$$\frac{1}{2\theta}\|u\|_{X_0}^{2\theta}-\int_{\Omega}\tilde{F}(u)dx>0,$$
		so sincee $r<1$, for $n$ sufficiently large, we have
		\begin{eqnarray}\label{aux1}
		\frac{1}{2\theta}\|v_n\|_{X_0}^{2\theta}-\int_{\Omega}\tilde{F}(v_n)dx>0.
		\end{eqnarray}
		
		Therefore by  \eqref{conv2}, \eqref{brez-lieb}, \eqref{aux1} and since $\theta>1$, we have
		\begin{align}\label{aux2}
		\begin{split}
		m_{\lambda}&=\bar{I}_{\lambda}(v_n)+o(1)\\
		 &=\frac{1}{2\theta}(\|v_n\|_{X_0}^{2\theta}+\|u_0\|_{X_0}^{2\theta})-\int_{\Omega}\frac{\lambda}{1-\gamma(x)}|u_0|^{1-\gamma(x)}dx-\int_{\Omega}(\tilde{F}(v_n)+\tilde{F}(u_0))dx+o(1)\\
		&= \bar{I}_{\lambda}(u_0)+\frac{1}{2\theta}\|v_n\|_{X_0}^{2\theta}-\frac{1}{\alpha^-}\int_{\Omega}\eta\left(\frac{\|v_n\|_{X_0}^2}{2}\right)|v_n|^{\alpha(x)}dx+o(1)\\
		&= \bar{I}_{\lambda}(u_0)+\frac{1}{2\theta}\|v_n\|_{X_0}^{2\theta}-\frac{1}{\alpha^-}\int_{\Omega}|v_n|^{\alpha(x)}dx+o(1);~\text{since}~\|v_n\|_{X_0}\leq r<1\\
		&\geq \bar{I}_{\lambda}(u_0)+o(1)>m_{\lambda}.
		\end{split}
		\end{align}
		
		The last inequality is due to $u_0\in\bar{B_r}$. Therefore, $u_0$ is a local minimizer of $\bar{I}_{\lambda}$ with $\bar{I}_{\lambda}(u_0)=m_{\lambda}<0$. This also implies that $u_0\neq 0$. Thus the minimization problem \eqref{min_prob} has been solved. Since $\bar{I}_{\lambda}$ is a $C^1$ functional in $X_0\setminus\{0\}$,  we have $\bar{I}'_{\lambda}(u_0)=0$ which satisfies Definition \ref{weak_soln_mod_prob}.
	\end{proof}

	\subsection{Genus and its properties}
	In this section  we shall  prove Lemma \ref{ps limit} which establishes that  functional $\bar{I}_{\lambda}$ fulfills the $(PS)_c$ condition for $c<c_*$ where $c_*>0$. Then we shall show that all  hypotheses of the Clark theorem are satisfied and this will eventually yield our Theorem \ref{mainthm1}.

 Rabinowitz \cite{Rabinowitz1986} introduced the definition of {\it genus} $G(\cdot)$. We shall list some of the properties of  genus for the reader's convenience.
	Let $\Gamma$ denote the family of all closed subsets of $\mathcal{B}\setminus\{0\}$ that are symmetric with respect to the origin.
	\begin{lemma}\label{lemma genus}
		Let $A, B\in\Gamma$. Then
		\begin{enumerate}
			\item $A\subset B\Rightarrow G(A)\leq G(B)$.
			\item Suppose that  $A$ and $B$ are homeomorphic via an odd map. Then $G(A)=G(B)$.
			\item $G(\mathbb{S}^{N-1})=N$, where $\mathbb{S}^{N-1}$ is the $(N-1)$-sphere in $\mathbb{R}^{N}$.
			\item $G(A\cup B)\leq G(A)+\gamma(B)$.
			\item $G(A)<\infty\Rightarrow G (A\setminus B)\geq G(A)-G(B)$.		
			\item For every compact subset $A$ of $\mathcal{B}$, $G(A)<\infty,$ and there exists $\delta>0$ such that $G(A)=G(N_{\delta}(A)),$ where $N_{\delta}(A)=\{x\in \mathcal{B}:d(x,A)\leq\delta\}.$		
			\item Suppose $Y\subset \mathcal{B}$ is a subspace of $\mathcal{B}$ such that $codim(Y)=k$ and $G(A)>k$. Then $A\cap \mathcal{B}\neq\emptyset.$
		\end{enumerate}	
	\end{lemma} 
 
 Define $\Gamma_n$ to be the collection of closed and symmetric subsets of $\mathcal{B}$ such that none of which includes $0$ in it and the genus of each of these subsets is at least $n$.

 We will need the following Clark theorem \cite{kajikiya2005}.	
	
	\begin{theorem}\label{sym mountain}
			Let $\mathcal{B}$ be an infinite-dimensional Banach space and
			suppose that $I\in C^1(\mathcal{B},\mathbb{R})$ obeys the following conditions:
		\begin{itemize}
			\item[(i)] $I(u)=I(-u)$ is bounded below, $I(0)=0,$ and $I$ fulfills the $(PS)_c$ condition.
			\item[(ii)] For any $n\in\mathbb{N}$, there exists  $A_n\in\Gamma_n$ such that $\sup\limits_{u\in A_n}I(u)<0.$
		\end{itemize}
		Then $I$ accomodates a sequence of critical points $\{u_n\}$ such that $I(u_n)\leq 0$, $u_n\neq 0$ and $u_n\rightarrow 0$ as $n\rightarrow\infty$.
	\end{theorem}
	  Another important {\it compactness} condition which will be used later is the Palais-Smale condition at energy level $c\in\mathbb{R}$, see  \cite[Definition 5.1.6]{repovs_et_al}.
	  We shall  now prove that $\bar{I}_{\lambda}$ satisfies the $(PS)_c$ condition under suitable range.
	\begin{lemma}\label{ps limit}
		There exists $\lambda_0>0$, such that for every $\lambda\in (0,\lambda_0)$ the functional $\bar{I}_{\lambda}$ fulfills  $(PS)_{c}$ condition, whenever $$c<\left(\frac{1}{2\theta}-\frac{1}{\beta}\right)S^{\alpha^*\theta/(\alpha^*-2\theta )}-\left(\frac{1}{2\theta}-\frac{1}{\beta}\right)^{-\frac{1-\delta}{2\theta-1+\delta}}\left[\lambda C\left(\frac{1}{1-\gamma^+}-\frac{1}{\beta}\right)\right]^{\frac{2\theta}{2\theta-1+\delta}}=c_*$$ where $c_*>0$.
	\end{lemma}
	\begin{proof}
		Since $\bar{I}_{\lambda}$ is not differentiable at $u=0$,  we first fix the sequences that will be considered in this lemma. Suppose $\{u_n\}\subset X_0$ is an eventually zero sequence, then it converges to $0$ and so we discard it immediately. Suppose $\{u_n\}\subset X_0$ is a sequence with infinitely many terms equal to $0$, then we choose a subsequence of $\{u_n\}$ with all terms nonzero. Thus, without loss of generality,  we shall  assume that $\{u_n\}$ is such that $u_n\neq 0$ for every $n\in\mathbb{N}$ and $u_n\nrightarrow 0$ in $X_0$ as $n\rightarrow \infty$. Let this sequence $\{u_{n}\}$ be such that
		\begin{align}\label{ps 1}
		\bar{I}_{\lambda}(u_{n})\rightarrow c~\text{and}~ \bar{I}_{\lambda}^{'}(u_{n})\rightarrow0,
		\quad
		\hbox{as}
		\quad n\rightarrow\infty. 
		\end{align}	
		
		We shall show that $\{u_n\}$ is bounded in $X_0$. Indeed, observe that $\bar{I}_{\lambda}$ is a coercive functional. This is because
		\begin{align}\label{coercive1}
		\begin{split}
        \underset{\|u\|_{X_0}\rightarrow\infty}{\lim}\frac{\bar{I}_{\lambda}(u)}{\|u\|_{X_0}}&=\underset{\|u\|_{X_0}\rightarrow\infty}{\lim}\left(\frac{1}{2}\|u\|_{X_0}^{2\theta-1}-\frac{1}{\|u\|_{X_0}}\int_{\Omega}\frac{|u|^{1-\gamma(x)}}{1-\gamma(x)}dx-\frac{1}{\|u\|_{X_0}}\int_{\Omega}\eta\left(\frac{\|u\|_{X_0}^2}{2}\right)\frac{|u|^{\alpha(x)}}{\alpha(x)}dx\right)\\
        &=\underset{\|u\|_{X_0}\rightarrow\infty}{\lim}\left(\frac{1}{2}\|u\|_{X_0}^{2\theta-1}-\frac{1}{\|u\|_{X_0}}\int_{\Omega}\frac{|u|^{1-\gamma(x)}}{1-\gamma(x)}dx\right),~\text{since by the definition of}~\eta\\
        &\geq\underset{\|u\|_{X_0}\rightarrow\infty}{\lim}\left(\frac{1}{2}\|u\|_{X_0}^{2\theta-1}-\frac{C'}{\|u\|_{X_0}}\frac{\|u\|_{X_0}^{1-\gamma ^-}}{1-\gamma^+}\right)=\infty.
        \end{split}
		\end{align}
		
		Moreover, since $\bar{I}_{\lambda}(u_{n})\rightarrow c$ as $n\rightarrow\infty$ for some $0<c<c_*$, the sequence $\{u_n\}$ is bounded. Suppose not, then there exists a subsequence of $\{u_n\}$, such that $\|u_n\|_{X_0}\rightarrow\infty$ as $n\rightarrow\infty$. Then from \eqref{coercive1} we obtain
		\begin{align}\label{coercive2}
		\begin{split}
		0&=\underset{\|u\|_{X_0}\rightarrow\infty}{\lim}\frac{\bar{I}_{\lambda}(u_n)}{\|u_n\|_{X_0}}\\
		 &\geq\underset{\|u\|_{X_0}\rightarrow\infty}{\lim}\left(\frac{1}{2}\|u_n\|_{X_0}^{2\theta-1}-\frac{C'}{\|u_n\|_{X_0}}\frac{\|u_n\|_{X_0}^{1-\gamma ^-}}{1-\gamma^+}\right)=\infty,
		\end{split}
		\end{align}
		which is absurd. Hence, we get the boundedness of the sequence $\{u_{n}\}$ in $X_0$. 
		
		Since the space $X_0$ is reflexive,  there exists a subsequence of $\{u_{n}\}$ and $u\in X_0$ such that as $n\rightarrow\infty$ we have
		\begin{align}\label{conv_D}\tag{D}
		\begin{split}
		&u_{n}\rightharpoonup u ~\text{weakly in}~ X_0,~\|u_n\|_{X_0}\rightarrow a,\\
		&u_{n}\rightarrow u ~\text{strongly in}~ L^{r(\cdot)}(\Omega)~\text{for}~ 1\leq r(x)<2_{s^-}^{*},\\
		&u_{n}(x)\rightarrow u(x) ~\text{a.e. in}~ \Omega,\\
		&u_{n}\rightharpoonup u ~\text{weakly in}~L^{\alpha(x)}(\Omega)~\text{if}~(A_2)~\text{is satisfied},\int_{\Omega}|u_n-u|^{\alpha(x)}dx\rightarrow b^{\alpha^*},
		\end{split}
		\end{align}
		where  \begin{equation}\label{delta_3}\alpha^*=\begin{cases}
		(1-t_0)\alpha^-+t_0\alpha^+, & b>1\\
		(1-t_0)\alpha^++t_0\alpha^-, & b<1
		\end{cases}\end{equation}
		for some $t_0\in(0,1)$ and $\|u_n-u\|_{\alpha(\cdot)}\rightarrow b$. If $a=0$, then it is an immediate conclusion that $u_n\rightarrow 0$ in $X_0$ and hence $u_n\rightarrow 0$ in $L^{\alpha(x)}(\Omega)$ as $n\rightarrow\infty$ (by Theorem \ref{embedding}). Therefore we discard the case of $a=0$.
		
		 Since by Remark \ref{key_obs} we are interested in solutions which obey $$\frac{\|u\|_{X_0}^2}{2}\leq l,$$  we further make the assumption $$0<\frac{a^2}{2}\leq l.$$ Note that this is where we are given a possible choice of $l$ in the definition of $\eta$. 
		 
		 Moreover, by the weak convergence we have
		\begin{align*}
		\lim\limits_{n\rightarrow+\infty}\iint_{Q}\frac{(u_{n}(x)-u_{n}(y))(v(x)-v(y))}{|x-y|^{N+2s(x,y)}}dxdy
		&=\iint_{Q}\frac{(u(x)-u(y))(v(x)-v(y))}{|x-y|^{N+2s(x,y)}}dxdy.
		\end{align*}
		
		Further, we have
		\begin{eqnarray}\label{conv sing}
		\lim\limits_{n\rightarrow+\infty}\int_{\Omega}|u_n|^{-\gamma(x)}vdx&=&
		\int_{\Omega}|u|^{-\gamma(x)}vdx<\infty.\end{eqnarray}
		
		The finiteness of the limit can be established by  Remark \ref{sing_conv_pf}.
		Now consider
		\begin{align}\label{PS_1_AUX}
		\begin{split}
		o(1)=&\langle \bar{I}'_{\lambda}(u_n),u_n-u \rangle\\
		=&\|u_n\|_{X_0}^{2(\theta-1)}\langle u_n, u_n-u\rangle-\lambda\int_{\Omega}|u_n|^{-\gamma(x)-1}u_n(u_n-u)dx-\int_{\Omega}\eta\left(\frac{\|u_n\|_{X_0}^2}{2}\right)|u_n|^{\alpha(x)-2}u_n(u_n-u)dx\\
		&-\eta'\left(\frac{\|u_n\|_{X_0}^2}{2}\right)\langle u_n,u_n-u\rangle\int_{\Omega}\frac{|u_n|^{\alpha(x)}}{\alpha(x)}dx
		\end{split}
		\end{align}
	\begin{align}
	\begin{split}
		=&\|u_n\|_{X_0}^{2(\theta-1)}\langle u_n, u_n-u\rangle-\lambda\int_{\Omega}|u_n|^{-\gamma(x)-1}u_n(u_n-u)dx-\int_{\Omega}\eta\left(\frac{a^2}{2}\right)|u_n|^{\alpha(x)-2}u_n(u_n-u)dx\\
		&-\eta'\left(\frac{a^2}{2}\right)\langle u_n,u_n-u\rangle\int_{\Omega}\frac{|u_n|^{\alpha(x)}}{\alpha(x)}dx\\
		=&\|u_n\|_{X_0}^{2(\theta-1)}(a^2-\|u\|_{X_0}^2)-\int_{\Omega}|u_n|^{\alpha(x)}dx+\int_{\Omega}|u|^{\alpha(x)}dx+o(1)\\
		=&\|u_n\|_{X_0}^{2(\theta-1)}\|u_n-u\|_{X_0}^2-\int_{\Omega}|u_n-u|^{\alpha(x)}dx+o(1).
		\end{split}
		\end{align}
		
		Note that, we have again used  Remark \ref{sing_conv_pf} to tackle the singular term and that we  get  $\underset{n\rightarrow\infty}{\lim}\int_{\Omega}uu_n^{-\gamma(x)}dx=\int_{\Omega}u^{1-\gamma(x)}dx$. Thus we derive
		\begin{align}\label{PS_2_AUX}
		\begin{split}
		 a^{2(\theta-1)}\underset{n\rightarrow\infty}{\lim}\|u_n-u\|_{X_0}^2&=\underset{n\rightarrow\infty}{\lim}\int_{\Omega}|u_n-u|^{\alpha(x)}dx=b^{\alpha^*}.
		\end{split}
		\end{align}
		
		If $b=0$, then from \eqref{conv_D} and \eqref{PS_2_AUX} we obtain $u_n\rightarrow u$ in $X_0$ as $n\rightarrow\infty$, since $a>0$. Therefore we need to show that $b=0$. Let us assume to the contrary, that $b>0$. Then,  using \eqref{sobolev_constant}, \eqref{conv_D}, and \eqref{PS_2_AUX}, we would get
		\begin{align}\label{PS_3_AUX}
		\begin{split}
		0\leq a^{2(\theta-1)}\underset{n\rightarrow\infty}{\lim}\|u_n-u\|_{X_0}^2&=b^{\alpha^*}.
		\end{split}
		\end{align}
		
		Next, from \eqref{PS_3_AUX} we would have
		\begin{align}\label{PS_4_AUX}
		\begin{split}
		Sa^{2(\theta-1)}b^2&\leq b^{\alpha^*},\\
		a^{2(\theta-1)}(a^2-\|u\|_{X_0}^2)&=b^{\alpha^*}.
		\end{split}
		\end{align}
		
		It would also follow from \eqref{PS_4_AUX} that
		\begin{align}\label{PS_5'_AUX}
		\begin{split}
		b&\geq S^{\frac{1}{\alpha^*-2}}a^{\frac{2(\theta-1)}{\alpha^*-2}}
		\end{split}
		\end{align}
		and from \eqref{PS_4_AUX} that
		\begin{align}\label{PS_5_AUX}
		\begin{split}
		a^2\geq Sb^2&\geq S(S^{\frac{2}{\alpha^*-2}}a^{\frac{4(\theta-1)}{\alpha^*-2}})\\
		&=S^{\frac{\alpha^*}{\alpha^*-2}}a^{\frac{4(\theta-1)}{\alpha^*-2}}.
		\end{split}
		\end{align}
		
		This would further lead to
		\begin{align}\label{PS_5''_AUX}
		\begin{split}
		a^{\frac{2(\alpha^*-2\theta)}{\alpha^*-2}}&\geq S^{\frac{\alpha^*}{\alpha^*-2}}~\text{which would imply that}\
		a^2\geq S^{\frac{\alpha^*}{\alpha^*-2\theta}}.
		\end{split}
		\end{align}
		
		We can assume that no subsequence of the bounded (PS) sequence $\{u_n\}$ is such that $\|u_n-u\|_{X_0}\rightarrow 0$, as $n\rightarrow\infty,$ whenever $0<c<c_*$. Let us further choose $\beta$ such that $1-\gamma^-<1<2\theta<\beta<\alpha^-$ and consider the following
		\begin{align}\label{PS_8_AUX}
		\begin{split}
		\bar{I}_{\lambda}(u_n)- \frac{1}{\beta}\langle\bar{I}'_{\lambda}(u_n),u_n\rangle\geq& \left(\frac{1}{2\theta}-\frac{1}{\beta}\right)\|u_n\|_{X_0}^{2\theta}+\lambda\left(\frac{1}{\beta}-\frac{1}{1-\gamma^+}\right)\int_{\Omega}|u_n|^{1-\gamma(x)}dx\\
		 &+\left(\frac{1}{\beta}-\frac{1}{\alpha^-}\right)\int_{\Omega}\eta\left(\frac{\|u_n\|_{X_0}^2}{2}\right)|u_n|^{\alpha(x)}dx+\frac{1}{\beta}\|u_n\|_{X_0}^2\eta'\left(\frac{\|u_n\|_{X_0}^2}{2}\right)\int_{\Omega}\frac{|u_n|^{\alpha(x)}}{\alpha(x)}dx\\
		 \geq&\left(\frac{1}{2\theta}-\frac{1}{\beta}\right)\|u_n\|_{X_0}^{2\theta}+\lambda\left(\frac{1}{\beta}-\frac{1}{1-\gamma^+}\right)\int_{\Omega}|u_n|^{1-\gamma(x)}dx\\
		 &+\frac{1}{\beta}\|u_n\|_{X_0}^2\eta'\left(\frac{\|u_n\|_{X_0}^2}{2}\right)\int_{\Omega}\frac{|u_n|^{\alpha(x)}}{\alpha(x)}dx.
		\end{split}
		\end{align}
		
		On passing to the limit $n\rightarrow\infty$ in \eqref{PS_8_AUX} and using $0<\frac{a^2}{2}\leq l$, equation \eqref{ps 1},
		  the Br\'{e}zis-Lieb Lemma, and the Young inequality, we would get
				\begin{align*}
		\begin{split}
		c\geq& \left(\frac{1}{2\theta}-\frac{1}{\beta}\right)(a^{2\theta}+\|u\|_{X_0}^{2\theta})-\lambda\left(\frac{1}{1-\gamma^+}-\frac{1}{\beta}\right)\int_{\Omega}|u|^{1-\gamma(x)}dx+\frac{1}{\beta}a^2\eta'\left(\frac{a^2}{2}\right)\underset{n\rightarrow\infty}{\lim}\int_{\Omega}\frac{|u_n|^{\alpha(x)}}{\alpha(x)}dx\\
		 \geq&\left(\frac{1}{2\theta}-\frac{1}{\beta}\right)(a^{2\theta}+\|u\|_{X_0}^{2\theta})-C\lambda\left(\frac{1}{1-\gamma^+}-\frac{1}{\beta}\right)\|u\|_{X_0}^{1-\delta}\\
		 \geq& \left(\frac{1}{2\theta}-\frac{1}{\beta}\right)(a^{2\theta}+\|u\|_{X_0}^{2\theta})-\left(\frac{1}{2\theta}-\frac{1}{\beta}\right)\|u\|_{X_0}^{2\theta}-\left(\frac{1}{2\theta}-\frac{1}{\beta}\right)^{-\frac{\delta}{2\theta-1+\delta}}\left[\lambda C\left(\frac{1}{1-\gamma^+}-\frac{1}{\beta}\right)\right]^{\frac{2\theta}{2\theta-1+\delta}}
		 \end{split}
		 \end{align*}
		 \begin{align}\label{PS_9_AUX}
		 \begin{split}
		\geq& \left(\frac{1}{2\theta}-\frac{1}{\beta}\right)a^{2\theta}-\left(\frac{1}{2\theta}-\frac{1}{\beta}\right)^{-\frac{\delta}{2\theta-1+\delta}}\left[\lambda C\left(\frac{1}{1-\gamma^+}-\frac{1}{\beta}\right)\right]^{\frac{2\theta}{2\theta-1+\delta}}\\
		\geq& \left(\frac{1}{2\theta}-\frac{1}{\beta}\right)S^{\alpha^*\theta/(\alpha^*-2\theta )}-\left(\frac{1}{2\theta}-\frac{1}{\beta}\right)^{-\frac{1-\delta}{2\theta-1+\delta}}\left[\lambda C\left(\frac{1}{1-\gamma^+}-\frac{1}{\beta}\right)\right]^{\frac{2\theta}{2\theta-1+\delta}}\\
		=&c_*.
		\end{split}
		\end{align}
		
		This is a contradiction to the assumption that $c<c_*$. Also, for a sufficiently small $\lambda$, say $\lambda_0$, one can get $c_*>0$. Therefore we conclude that $b=0$.
		Thus we have obtained $u_n\rightarrow u$ in $X_0$ as $n\rightarrow\infty$. So the functional $\bar{I}_{\lambda}$ fulfills the $(PS)_c$ condition whenever $c<c_*$.
	\end{proof}
	\begin{remark}\label{positive_ps_energy}
		Observe that since $c_*>0$, we can use  Theorem \ref{sym mountain}.
	\end{remark}
	\begin{remark}\label{scheme of seq}
	Notice that since $c_*>0$, there is a possibility of the occurrence of the situation where the sequence obeys the following
\begin{align}\label{zero ps}
\bar{I}_{\lambda}(u_n)\rightarrow 0,~\bar{I}'_{\lambda}(u_n)\rightarrow 0,
\end{align}	
as $n\rightarrow\infty$. This implies that there exists a subsequence of $\{u_n\}$ such that $u_n\rightarrow v_0$ in $X_0$ as $n\rightarrow\infty$. This forces us to have $v_0\neq 0$ a.e. in $\Omega$. For if $v_0=0$ over a subset of $\Omega$ of nonzero Lebesgue measure, then we shall  have $
\bar{I}_{\lambda}(u_n)\rightarrow 0,~\bar{I}'_{\lambda}(u_n)\rightarrow -\infty,$ which is a contradiction to \eqref{zero ps}.	
	\end{remark}	
	
	\section{Proofs of the main results}
	
\subsection{Proof of Theorem \ref{embedding}}
		An important embedding result in the literature is due to Ho and Kim \cite{HK}, who considered a Sobolev space of constant order but with a variable exponent. We will prove the embedding result for a variable order but with a constant exponent.

					 Using the Closed graph theorem,	
			  we only need to establish
			   that $W^{s(\cdot),2}(\Omega)\subset L^{\alpha(\cdot)}(\Omega)$.
			   To this end, let $u\in W^{s(\cdot),2}(\Omega)\backslash\{0\}$ be arbitrary and fixed. We shall  show that $u\in L^{\alpha(\cdot)}(\Omega)$, namely,
			\begin{equation}\label{j}
			\int_{\Omega}|u|^{\alpha(x)}dx<\infty.
			\end{equation}
			
			Throughout the proof, $K>0$  will denote different constants independent of $u$.
			We cover $\overline{\Omega}$ by $\{B_{\varepsilon_{j}}(x_{j})\}_{j=1}^{k}$ with $x_{j}\in \Omega$ and $\varepsilon_{j}\in (0,1)$. Here $\Omega_{j}:=B_{\varepsilon_{j}}(x_{j})\bigcap\Omega$ are Lipschitz domains and the condition in $(A_{2})$ is satisfied for all $j\in\{1,...,k\}$.
			We fix $j\in\{1,...,k\}$ and denote $s_{j}:=\inf_{(y,z)\in\Omega_{j}\times\Omega_{j}}s(y,z)$,  $\alpha_{j}:=\sup_{x\in\Omega_{j}}\alpha(x).$ From $(A_{2})$ and the choice of $\varepsilon_{j}$, we get
			$$\alpha_{j}\leq\frac{2N}{N-2s_{j}}:=2_{s,j}^{*}.$$
			
			By the above expression and the well known embedding theorem in the constant exponent case (refer \cite{11}), we obtain
			\begin{eqnarray*}
				\int_{\Omega_{j}}|u(x)|^{\alpha_{j}}dx\leq K\left(\int_{\Omega_{j}}|u(x)|^{2}dx+\iint_{\Omega_{j}\times\Omega_{j}}\frac{|u(x)-u(y)|^{2}}{|x-y|^{N+2s_{j}}}dxdy\right)^{\frac{\alpha_{j}}{2}}.
			\end{eqnarray*}

			Besides,  we note that
			\begin{align}\label{c2}
			 \iint_{\Omega_{j}\times\Omega_{j}}&\frac{|u(x)-u(y)|^{2}}{|x-y|^{N+2s(x,y)}}dxdy\geq\iint_{\Omega_{j}\times\Omega_{j}}\frac{|u(x)-u(y)|^{2}}{|x-y|^{N+2s_{j}}}dxdy\nonumber\\
			&=\iint_{\Omega_{j}\times\Omega_{j}}\left|\frac{|u(x)-u(y)|}{|x-y|^{2s_{j}}}\right|^{2}\frac{1}{|x-y|^{N-2s_{j}}}dxdy.
			\end{align}

A simple calculation leads to
			\begin{align*}
			\int_{\Omega_{j}}|u(x)|^{\alpha(x)}dx&\leq K\left(1+\int_{\Omega_{j}}|u(x)|^{2}dx+\iint_{\Omega_{j}\times\Omega_{j}}\frac{|u(x)-u(y)|^{2}}{|x-y|^{N+2s(x,y)}}dxdy\right)^{\frac{\alpha_{j}}{2}}\\
			&\leq K\left(1+\int_{\Omega}|u(x)|^{2}dx+\iint_{\Omega\times\Omega}\frac{|u(x)-u(y)|^{2}}{|x-y|^{N+2s(x,y)}}dxdy\right)^{\frac{\alpha^{+}}{2}}.
			\end{align*}
			
			On summing up over the index $j$ we get
			\begin{align*}
			\int_{\Omega}|u(x)|^{\alpha(x)}dx\leq K\left(1+\int_{\Omega}|u(x)|^{2}dx+\iint_{\Omega\times\Omega}\frac{|u(x)-u(y)|^{2}}{|x-y|^{N+2s(x,y)}}dxdy\right)^{\frac{\alpha^{+}}{2}}<\infty,
			\end{align*}
			and so \eqref{j} holds. Therefore $\|u\|_{\alpha(\cdot)}\leq C\|u\|_{W^{s(\cdot),2}(\Omega)}$ holds.
		\qed 

\subsection{Proof of Theorem \ref{mainthm1}}

   We now prove that problem \eqref{main3} has infinitely many solutions.
	
\begin{proof}[Proof of Theorem \ref{mainthm1}]
For each $m\in\mathbb{N}$, choose $Y_m\subset X_0$, a finite-dimensional subspace of dimension $m$. Therefore, by $(2)$ and $(3)$ of Lemma \ref{lemma genus}, we have that $G(\mathbb{S}^{m-1})=m$ and therefore $G(r\mathbb{S}^{m-1})=m$.

 Choosing a sufficiently small $r$, say $r_0=r_0(Y_m)$, such that for every $u\in r_0\mathbb{S}^{m-1}$, i.e. $\|u\|_{X_0}=r_0<\sqrt{2l}$, according to Theorem \ref{embedding}, we have
	\begin{align}\label{sym_hypo}
	\begin{split}
	 \bar{I}_{\lambda}(u)\leq&\frac{1}{2\theta}\|u\|_{X_0}^{2\theta}-\frac{\lambda}{1-\gamma^-}\int_{\Omega}|u|^{1-\gamma(x)}dx-\frac{1}{\alpha^+}\int_{\Omega}\eta\left(\frac{\|u\|_{X_0}^2}{2}\right)|u|^{\alpha(x)}dx\\
    =&\frac{1}{2\theta}\|u\|_{X_0}^{2\theta}-\frac{\lambda}{1-\gamma^-}\int_{\Omega}|u|^{1-\gamma(x)}dx-\frac{1}{\alpha^+}\int_{\Omega}|u|^{\alpha(x)}dx\\
	 \leq&\frac{1}{2\theta}\|u\|_{X_0}^{2\theta}-\frac{C''\lambda}{1-\gamma^-}\|u\|_{X_0}^{1-\gamma^-}-\frac{C'''}{\alpha^+}\|u\|_{X_0}^{\alpha^+}dx<0,~\text{due to the ordering}~1-\gamma^-<2\theta< \alpha^-.
	\end{split}
	\end{align}
	
	The hypotheses of Theorem \ref{sym mountain}, namely  $(i)-(ii)$ are all satisfied. Note that the hypothesis $(i)$ holds since $\bar{I}_{\lambda}(0)=0$, $\bar{I}_{\lambda}$ is even. The hypothesis $(ii)$ holds by the discussion leading to \eqref{sym_hypo}. Thus, by Theorem \ref{sym mountain} we establish the existence of infinitely many critical points $\{u_m\}$ for functional $\bar{I}_{\lambda}$ which in turn proves, by Remark \ref{key_obs},  that problem \eqref{main3} has infinitely many solutions, and hence we have established the existence of infinitely many solutions for problem \eqref{problem}.
	\end{proof}
	
	\subsection{Proof of Theorem \ref{bdd_thm}}
	  Finally, we address the boundedness of  solutions to problem \eqref{problem}.
	\begin{proof}[Proof of Theorem \ref{bdd_thm}]
	The aim of this subsection is to establish the boundedness of the solution to the problem \eqref{problem}. We shall sketch the ideas without giving the details as the argument is pretty standard one. We shall use a {\it bootstrap} argument in order to achieve this by assuming the integrability of certain order $p>1$, to begin with. 
	
	Let $0 \leq u\leq X_0$ be a solution such that $|\{x:u(x)=0\}|=0$.
	We consider the set $\Omega'=\{x\in\Omega:u(x)>1\}$ and thus by the positivity of $u$, we have $u=u^+>0$ a.e. in $\Omega$. Let $u\in L^{p}(\Omega)$ for $p>1$. On testing with $u^p$ in \eqref{weak_soln} we obtain the following
	\begin{align}\label{bddness}
	\begin{split}
	&\frac{(p+1)^2}{4p}\|u\|_{X_0}^{2(\theta-1)}\langle u,u^{p}\rangle,~\text{(see ~\cite[Lemma C.2]{Brasco})}\\
	&=\left(\lambda\int_{\Omega'}|u|^{p-\gamma(x)}dx+\int_{\Omega'}|u|^{\alpha(x)-1+p}dx\right)\frac{(p+1)^2}{4p}\\
	&\leq\left(\lambda\int_{\Omega'}|u|^{p}(1+|u|^{\alpha(x)-1})dx\right)\frac{(p+1)^2}{4p};~\text{since in}~\Omega'~\text{we have}~u>1\\
	&\leq\left(\lambda\int_{\Omega'}|u|^{p}|u|^{\alpha^+}dx\right)\frac{(p+1)^2}{4p}\\
	&\leq\lambda p C''\|u\|_{\beta^*}^{\alpha^+}\|u^p\|_t;~\text{due to H\"{o}lder's inequality},
	\end{split}
	\end{align}
	where $t=\frac{\beta^*}{\beta^*-\alpha^+}$ for some $\beta^*>\alpha^+>1$.
	We further have
	\begin{align}\label{bddness1}
	\begin{split}
	C'\|u\|_{X_0}^{2(\theta-1)}\|u^{\frac{p}{2}}\|_{\beta^*}^2&\leq C'\|u\|_{X_0}^{2(\theta-1)}\|u^{\frac{p+1}{2}}\|_{\beta^*}^2\\
	&\leq C'\|u\|_{X_0}^{2(\theta-1)}\iint_{\Omega'\times\Omega'}\frac{\left|u(x)^{\frac{(p+1)}{2}}-u(y)^{\frac{(p+1)}{2}}\right|^2}{|x-y|^{N+2s(x,y)}}dxdy.
	\end{split}
	\end{align}
	
	By \eqref{bddness} and \eqref{bddness1}, we have the following
	\begin{align}
	C'\|u\|_{X_0}^{2(\theta-1)}\|u^{\frac{p}{2}}\|_{\beta^*}^2&\leq \lambda p C''\|u\|_{\beta^*}^{\alpha^+}\|u^p\|_t.
	\end{align}
	
We fix $\beta^*>1$, set $\Theta=\frac{\beta^*}{2t}>1,$ and pick a suitable $t$ and $\tau=tp$ to get
	\begin{align}\label{moser1}
	\|u\|_{\Theta\tau}&\leq (pC)^{t/\tau}\|u\|_{\tau},
	\end{align}
	where $C=\frac{\lambda  C''\|u\|_{\beta^*}^{\alpha^+}}{\|u\|^{2(\theta-1)}}$  corresponds to the chosen solution $u$. 
	
	We now iterate with $\tau_0=t$, $\tau_{n+1}=\Theta\tau_n=\Theta^{n+1}t$. After $n$ iterations, the inequality \eqref{moser1} yields
	\begin{align}\label{moser2}
	\|u\|_{\tau_{n+1}}&\leq C^{\sum\limits_{i=0}^{n}\frac{t}{\tau_i}}\prod\limits_{i=0}^{n}\left(\frac{\tau_i}{t}\right)^{\frac{t}{\tau_i}}\|u\|_t.
	\end{align}
	
	By using the fact that $\Theta>1$ and employing the following {\it scheme} of iteration: $\tau_0=t$, $\tau_{n+1}=\Theta\tau_n=\Theta^{n+1}t$,
	we have $$\sum\limits_{i=0}^{\infty}\frac{t}{\tau_i}=\sum\limits_{i=0}^{\infty}\frac{1}{\Theta^i}=\frac{\Theta}{\Theta-1}$$ and
	$$\prod\limits_{i=0}^{\infty}\left(\frac{\tau_i}{t}\right)^{\frac{t}{\tau_i}}=\Theta^{\frac{\Theta^2}{(\Theta-1)^2}}.$$
	
	Therefore, on passing to the limit $n\rightarrow\infty$ in \eqref{moser2}, we have
	\begin{align}\label{moser3}
	\|u\|_{\infty}&\leq C^{\frac{\Theta}{\Theta-1}}\Theta^{\frac{\Theta^2}{(\Theta-1)^2}}\|u\|_t,
	\end{align}
	hence $u\in L^{\infty}(\Omega)$.
	\end{proof}
\section{Epilogue}
	We shall complete the paper with the following observations.

	\begin{remark}\label{sing_conv_pf}		
A similar remark as in \cite{chou_saoud} can also be made here in order to establish the following exchanging of the integral and the limit for the singular term.
		\begin{align}\label{conv_1}
\begin{split}
		\lim\limits_{n\rightarrow+\infty}\int_{\Omega}|u_n|^{-\gamma(x)-1}uvdx=\int_{\Omega}|u|^{-\gamma(x)-1}uvdx~\text{for every}~v\in X_0,
\end{split}
		\end{align}
		where $u$ is a pointwise limit of $\{u_n\}$, a.e. in $\Omega$.
\end{remark}
		
		As mentioned in  Section \ref{3.1} after  Definition \ref{defn_weak_soln_main_prob},  we shall  give the derivative of the singular term of $I_{\lambda}$ in a more clear way.

		\begin{remark}\label{derivative_explanation}
		The derivative of the functional
		$$J(u)=\int_{\Omega}|u|^{1-\gamma(x)}dx$$
		for a nonzero $u\in X_0$ is given by
		$$\langle J'(u),\varphi\rangle=\int_{\Omega}(1-\gamma(x))u|u|^{-1-\gamma(x)}\varphi dx~\text{for every}~\varphi\in X_0.$$
	\end{remark}
		  Indeed, we have $$\langle J'(u),\varphi\rangle=\int_{\Omega}(1-\gamma(x))|u|^{-\gamma(x)}\sgn(u)\varphi dx~\text{for every}~\varphi\in X_0.$$ 
		  
		  However, we also have
		\begin{align}
		\begin{split}
		\int_{\Omega}(1-\gamma(x))|u|^{-\gamma(x)-1}|u|\sgn(u)\varphi dx&=\int_{\Omega}(1-\gamma(x))|u|^{-\gamma(x)-1}u(\sgn(u))^2\varphi dx\\
		&=\int_{\Omega}(1-\gamma(x))|u|^{-\gamma(x)-1}u\varphi dx~\text{for every}~\varphi\in X_0.
		\end{split}
		\end{align}
		
		Therefore,
		\begin{align}
		\langle J'(u),\varphi\rangle&=\int_{\Omega}(1-\gamma(x))|u|^{-\gamma(x)-1}u\varphi dx~\text{for every}~\varphi\in X_0.
		\end{align}\qed

	\section*{Acknowledgements}
	  Choudhuri thanks the Department of Mathematics, NIT Rourkela for the basic facilities received and acknowledges the financial assitanceship from {\em Council of Scientific and Industrial Research - India} (25(0292)/18/EMR-II) and {\em Science and Engineering Research board (MATRICS scheme) - India} (MTR/2018/000525).
 Repov\v{s} acknowledges the support by the {\sl Slovenian Research Agency} through
  grants P1-0292, N1-0114, N1-0083, N1-0064, and J1-8131.
We thank the referees for their constructive comments and suggestions.

\end{document}